\documentclass[a4paper,12pt,twoside]{article}
\usepackage{amsmath,amssymb,ifthen}
\usepackage[amsmath,thmmarks]{ntheorem}
\usepackage{graphicx}
\usepackage{paper}
\usepackage[all]{xy}
\usepackage{mathrsfs}
\usepackage{url}

\DeclareMathOperator{\Ind}{Ind}
\DeclareMathOperator{\Res}{Res}

\newcommand{\G}{\mathrm{GSp}_{2n}}

\numberwithin{equation}{section}

\SelectTips{cm}{11}
\title{Correspondence of discrete series representations of $\mathrm{GSp}_{2n}$ and its inner form}
\author{Kaito Masuzawa}

\begin{document}
\maketitle



\begin{abstract}
	We construct a correspondence with a character equation of discrete series representations of $\mathrm{GSp}_{2n}$ over a $p$-adic field and its inner form for $p>2$ and $n\geq 2$.  This is a generalization of the works by Chan and Gan in \cite{MR3267112} for $n=2$ and the local Jacquet-Langlands  correspondence.  Moreover, our results partially reveal how the local Langlands correspondence of $\mathrm{GSp}_{2n}$ and its inner form over $p$-adic fields should be.  
\end{abstract}

\section{Introduction}

Let $F$ be a $p$-adic field and $D$ be a quaternion algebra over $F$.  The local Jacquet-Langlands correspondence is a bijection 
\[
	\mathrm{JL}\colon \,\Pi_{\mathrm{disc}}(\mathrm{GL}_2(F))\rightarrow \Pi_{\mathrm{disc}}(D^\times)
\]
satisfying various properties which is constructed in \cite{MR0401654}.  Here, $\Pi_{\mathrm{disc}}$ denotes the set of isomorphism classes of irreducible discrete series representations over $\C$.  

Now we explain one of the significant properties of this correspondence, that is the character relation.  There is a canonical correspondence between regular semisimple conjugacy classes of the two groups, $\mathrm{GL}_2(F)$ and $D^\times$.  Via this correspondence, for every corresponding regular semisimple elements $g\in\mathrm{GL}_2(F)$ and $h\in D^\times$, and for each $\pi\in\Pi_{\mathrm{disc}}(\mathrm{GL}_2(F))$, the equality
\[
	\Theta_{\pi}(g)=-\Theta_{\mathrm{JL}(\pi)}(h)
\]
holds.  Here, $\Theta_\pi$ is the character of $\pi$.  The local Jacquet-Langlands correspondence is generalized to $\mathrm{GL}_n$ and its inner forms over $p$-adic fields by Deligne, Kazhdan, and  Vign\'{e}ras in \cite{MR0771672}, and  to those over function fields by Badulescu in \cite{MR1951441}.    Moreover, these correspondences satisfy various properties which are enough for us to say they are canonical.  For details, see Section 2.2 in \cite{MR3618046}.  

Our main result is its generalization to the quasi-split general symplectic group $\mathrm{GSp}_{2n}$ and its inner form over $F$ with $p>2$ and $n\geq 2$.  We note that the group $\mathrm{GSp}_{2n}(F)$ can be written as 
\[
	\mathrm{GSp}_{2n}(F)=\{ g\in\mathrm{GL}_{2n}(F)\mid {}^t gjg=\mu(g)j,\;\mu(g)\in F^\times\}
\]
where $j\in\mathrm{GL}_{2n}(F)$ is the anti-diagonal matrix whose $(i, 2n+1-i)$ entry is $(-1)^{i+1}$.  Its unique non-quasi-split inner form $H$ is the quaternionic general unitary group. The group of $F$-points is given by
\[
	H(F)=\mathrm{GU}_n(D)=\{h\in\mathrm{GL}_n(D)\mid {}^t\overline{h} Jh=\mu(h)J,\;\mu(h)\in F^\times\}.  
\]
Here, $J\in\mathrm{GU}_n(D)$ is the anti-diagonal matrix whose $(i,n+1-i)$ entry is $1$ and $h\mapsto \overline{h}$ is an anti-involution of $D$ as an $F$-algebra.  

For $n=2$, Chan and Gan (\cite{MR3267112}) discovered the equality, multiplied by Kottwitz sign, between sums of characters of representations in $L$-packets of $\mathrm{GSp}_4(F)$ and those of $\mathrm{GU}_2(D)$.  The $L$-packets of a $p$-adic group $G$ are the fibers of the surjective finite-to-one map, from the set of isomorphism classes of irreducible admissible representations to the set of equivalence classes of $L$-parameters, $\Pi(G)\rightarrow \Phi(G)$ which is called the local Langlands correspondence.  The  local Langlands correspondence for general linear groups and their
 inner forms are bijective, that is, each $L$-packet consists of a single representation.  Therefore, the local Jacquet-Langlands correspondence can be viewed as a correspondence of $L$-packets with a character relation.  

For $G=\mathrm{GSp}_4(F)$, Gan and Takeda constructed the local Langlands correspondence in \cite{MR2800725}, while for $G=\mathrm{GU}_2(D)$, Gan and Tantono established it in \cite{MR3214276}.  These constructions of the local Langlands correspondence are due to their small rank.  Therefore, it is not easy to generalize their results to $n\geq 3$, and so far, the local Langlands correspondence for $\mathrm{GSp}_{2n}(F)$ is not clearly known.  On the other hand, by recent works by Xu (\cite{MR3568940} and \cite{MR3747484}), a natural method to partition $\Pi_{\mathrm{disc}}(\mathrm{GSp}_{2n}(F))$ into finite subsets is known. This is compatible with the local Langlands correspondence of $\mathrm{Sp}_{2n}(F)$ established by Arthur in \cite{MR3135650} via restriction of representations of $\mathrm{GSp}_{2n}(F)$ to those of $\mathrm{Sp}_{2n}(F)$.  For details, see Section 3.  

Although local $L$-packets proposed by Xu's work are not attached to $L$-parameters, we associate each Xu's local discrete $L$-packets (i.e. $L$-packets consists of discrete series representations) of $\mathrm{GSp}_{2n}(F)$ to a finite subset of $\Pi_{\mathrm{disc}}(\mathrm{GU}_n(D))$ with the character relation (Theorem \ref{main}).  This is our main result.  We note that Xu's construction of local $L$-packets are valid for all tempered representations, but in this paper, we only treat discrete series representations.  Due to this character relation, the sum of characters in each discrete $L$-packets of $\mathrm{GSp}_{2n}(F)$ are stable, that is, its value only depends on stable conjugacy classes.  Moreover, we can see these packets, also those of $\mathrm{GU}_n(D)$, are minimal sets in which some linear combination of characters of representations are stable functions (Section 4.3).  Kaletha calls this property atomically stable (cf. \cite{MR4680348}, see also \cite{some-comment}) and local $L$-packets of a $p$-adic group $G$ are highly expected to satisfy this property.  In addition, thanks to linear independence of characters of representations, this property might be enough to characterize the partition of $\Pi_{\mathrm{disc}}(G)$.  

Now we sketch the strategy to prove our main result.  We use a global method, the theory of the stable trace formula.  This framework to construct the correspondence is inspired by the argument in \cite[Section 11]{MR3267112}.  First of all, we globalize a discrete series representation of $\mathrm{GSp}_{2n}(F)$ to a cuspidal automorphic representation.  In this step, we use the Plancherel density theorem by Shin (\cite{MR3004076}) and this is applicable to discrete series representations.  Picking proper global test functions and applying the stable trace formula, we can get an equation between sums of trace distributions of some representations of $\mathrm{GSp}_{2n}(F)$ and $\mathrm{GU}_{n}(D)$.  Here, to the $\mathrm{GSp}_{2n}$ side, only some weakly equivalent classes contribute which is written as $\Sigma$ in this paper (introduced at the end of Section 2.2).  These are explained in Section 2.  

In Section 3, we give a brief review of Xu's local $L$-packets and global results.  At the end, we prove the finiteness of $\Sigma$.  In the next section, we show that $\Sigma$ coincide with a Xu's local $L$-packet.  This nontrivial result is due to a recent work by Kret and Shin in \cite{MR4556781}.  This leads us to the character relations of Xu's packets and some packets of the inner form.  In the following sections, we remark the atomical stability of packets of both sides and the case of simple supercuspidal packets.  

\vspace{15pt}
{\it Acknowledgements.}
The author is grateful to his supervisor Yoichi Mieda for his support.  He always gave effective advice whenever the author asked him various questions.  This paper was completed thanks to his precise guidance.  

\section{Stable trace formula}
\subsection{Globalization of discrete series representations of $\mathrm{GSp}_{2n}$}
Let $F$ be a $p$-adic field, $n\geqq 2$ be an integer, and $\pi_G$ be an irreducible discrete series representation of $\G(F)$.  In this paper we assume that the residual characteristic $p$ of $F$ is greater than $2$.  We take a totally real number field $L$ and its distinct three finite places $v_0,v_1,v_2$ such that the completion $L_{v_0}$ of $L$ at $v_0$ is $F$.  We write $\A_L$ for the adele ring of $L$.  

We start with the possibility of globalizing each discrete series representation to a certain cuspidal automorphic representation.  Now, we fix Haar measures on $\G(L_v)$ and $H(L_v)$ for each place $v$.  

\begin{prop}
	There exist a cuspidal automorphic representation $\Pi$ of $\G(\A_L)$ and unramified characters $\mu_i\colon \G(L_{v_i})\rightarrow\C^\times$ (i.e. the composites of the similitude character and an unramified character of $L_{v_i}^\times$) for $i=0,1$ satisfying 
	\begin{enumerate}
		\item $\Pi_{v_0}\cong \mu_0\otimes\pi_G$, 
		\item $\Pi_{v_1}\cong \mu_1\otimes\mathrm{St}_{v_1}$ where $\mathrm{St}_{v_1}$ is the Steinberg representation of $\G(L_{v_1})$, 
		\item $\Pi_v$ is unramified for any finite place $v$ other than $v_0,v_1, \text{and}\; v_2$, and 
		\item $\Pi_{v_\infty}$ is a discrete series representation for any infinite place $v_\infty$.
	\end{enumerate}
\end{prop}

\begin{prf}
	We use the same notation and prove the proposition by a similar argument in \cite[Theorem 5.7]{MR3004076}.  We put $S=\{v_0,v_1\}$ and let $\widehat{U}$ be the orbit of $\pi_G\otimes\mathrm{St}_{v_1}$ under twists by unramified characters of $\G(L_S)$.  As remarked in \cite[Example 5.6]{MR3004076},  the set $\widehat{U}$ is a $\widehat{\mu}^\mathrm{pl}_S$-regular relatively quasi-compact subset of the set of irreducible representations of $\G(L_S)$ such that $\widehat{\mu}^\mathrm{pl}_S(\widehat{U})>0$.  

Now we suppose this proposition is false.  Let $\{U_{v_2,n}\}_n$ be a decreasing sequence forming a fundamental system of neighborhoods of $1$ in $\G(L_{v_2})$ and $U^{S,v_2,\infty}$ be a hyperspecial maximal compact subgroup of $\G(L^{S,v_2,\infty})$.  We put $U_n=U_{v_2,n}U^{S, v_2,\infty}$.  If an irreducible algebraic representation $\xi$ of $(\mathop{\mathrm{Res}}_{L/\Q}\G)\times_\Q \R$ has a regular highest weight, by the equation (3.3) in \cite{MR3004076}, for every $\pi_S^0\in\widehat{U}$, we have the equality
\[
	m_\mathrm{cusp}(\pi_S^0, \mathrm{char}_{U_n}, \xi)=\sum_\pi m_\mathrm{cusp}(\pi)\mu^{S,\infty}(U_n)\dim(\pi^{S,\infty})^{U_n}\Tr\pi_\infty(\phi_\xi).  
\]
Here, $\phi_\xi$ denotes the Euler-Poincar\'{e} function (see the equation (2.1) in \cite{MR3004076}), and $\pi$ runs over the isomorphism classes of irreducible admissible representations of $\G(\A_L)$ such that $\pi_S\cong\pi_S^0$ and $\pi_\infty$ is $\xi$-cohomological.  Our assumption and \cite[Lemma 2.7]{MR3004076} show that each term of the right hand side of the above equation is zero.  Therefore, by the equation (3.5) in \cite{MR3004076}, we have 
\[
	\widehat{\mu}^\mathrm{cusp}_{\mathrm{char}_{U_n}, \xi}(\mathrm{char}_{\widehat{U}})=0
\]
for each $n$, hence $\widehat{\mu}^\mathrm{pl}_S (\widehat{U})=0$ by \cite[Corollary 4.5]{MR3004076}.  This contradicts to $\widehat{\mu}^\mathrm{pl}_S(\widehat{U})>0$.  This proves the proposition.  
\end{prf}

In the following of this paper, we fix such a pair $(\Pi,\mu_0,\mu_1)$.  Let $B$ be a quaternion algebra over $L$ which is ramified precisely at $v_0$ and $v_1$ and $H$ be the similitude group associated to a quaternionic Hermitian space over $B$ of split rank $n$.  In this situation, we can take isomorphisms $H(L_v)\cong \G(L_v)$ for any $v\neq v_0,v_1$.  In the followings, we fix these isomorphisms.  Now, we can take a test function $f=\otimes_v^\prime f_v$ on $H(\A_L)$ such that 
	\begin{itemize}
		\item $f_{v_1}$ is a product of $\mu_1$ and the pseudo-coefficient of the Steinberg representation $\mathrm{St}^H_{v_1}$ of $H(L_{v_1})$, 
		\item $f_{v_\infty}$ is the pseudo-coefficient of $\Pi_{v_\infty}$ via the above isomorphism $H(L_{v_\infty})\cong \G(L_{v_\infty})$ for any infinite $v_\infty$, and 
		\item $f_v$ is any element of the spherical Hecke algebra of $H(L_v)$ for every finite $v\neq v_1$.  
	\end{itemize}
We note that the character $\mu_i$ can be considered as an unramified character of $H(L_{v_i})$ via the similitude character and the same unramified character of $L_{v_i}^\times$ (i=0,1).  

Since $H(L_{v_i})$ is an inner form of $\G(L_{v_i})$ for $i=0,1$, we can also take test functions $f^G=\otimes_v^\prime f_v^G$ on $\G(\A_L)$ such that 
	\begin{itemize}
		\item $f_{v_1}^G$ is a product of $\mu_1$ and the pseudo-coefficient of $\mathrm{St}_{v_1}$, 
		\item $f_{v_\infty}^G$ is the pseudo-coefficient of $\Pi_{v_\infty}$ for any infinite $v_\infty$, 
		\item $f_v^G=f_v$ via the isomorphism $H(L_v)\cong \G(L_v)$ for every finite $v\neq v_0, v_1$, and 
		\item $f_{v_0}^G$ is a transfer of $f_{v_0}$ (see \cite[Section 4]{MR3267112}).  
	\end{itemize}

Then, due to the following lemma, we can see $f^G$ and $f$ are matching test functions.  

\begin{lem}\label{matching-orbital}
	The pseudo-coefficients of $\mathrm{St}_{v_1}$ and $\mathrm{St}^H_{v_1}$ have matching stable orbital integrals.  
\end{lem}

\begin{prf}
	Let $dg$ (resp. $dh$) be a Haar measure of $\G(L_{v_1})$ (resp. $H(L_{v_1})$) and $f_G$ (resp. $f_H$) be the pseudo-coefficient of $\mathrm{St}_{v_1}$ (resp. $\mathrm{St}_{v_1}^H$).  Moreover, let $f_{G,\mathrm{EP}}$ (resp. $f_{H,\mathrm{EP}}$) be the Euler-Poincar\'{e} function on $\G(L_{v_1})$ (resp. $H(L_{v_1})$) (cf. \cite{MR0942522}).  

Then, we have $f_G=(-1)^{r(\G(L_{v_1}))}f_{G,\mathrm{EP}}$ and $f_H=(-1)^{r(H(L_{v_1}))}f_{H,\mathrm{EP}}$ by \cite[Theorem $2^\prime$]{MR0942522}.  Here, $r$ denotes the split rank of the derived group.  So the Kottwitz sign of $H(L_{v_1})$ is equal to $(-1)^{r(\G(L_{v_1}))-r(H(L_{v_1}))}$.  

Now we take a strongly regular semisimple element $\gamma\in\G(L_{v_1})$.  We write $I_\gamma$ for the centralizer of $\gamma$ in $\G(L_{v_1})$ and $\mu_{I_\gamma}$ for the Euler-Poincar\'{e} measure on $I_\gamma$ (cf. \cite{MR0942522}).  Unless $\gamma$ is elliptic, the orbital integral of $f_G$ at $\gamma$ is zero by \cite[Theorem 2]{MR0942522}, similarly for $f_H$.  

Suppose that $\gamma$ is elliptic.  Then, there exists an elliptic regular semisimple element $\delta$ of $H(L_{v_1})$ which is stably conjugate to $\gamma$.  Let $J_\delta$ be the centralizer of $\delta$ in $H(L_{v_1})$ and $\mu_{J_\delta}$ be the Euler-Poincar\'{e} measure of $J_\delta$.  Since $\gamma$ and $\delta$ are elliptic, $r(I_\gamma)=r(J_\delta)=0$. Hence by \cite[Theorem 2]{MR0942522}, for invariant measures $di=(-1)^{r(\G(L_{v_1}))}\mu_{I_\gamma}$ and $dj=(-1)^{r(H(L_{v_1}))}\mu_{J_\delta}$, we have
\begin{align*}
	\int_{I_\gamma\backslash \G(L_{v_1})}f_G(g^{-1}\gamma g)dg/di&=(-1)^{r(\G(L_{v_1}))}\;\text{and}\\
	\int_{J_\delta\backslash H(L_{v_1})}f_H(h^{-1}\delta h)dh/dj&=(-1)^{r(H(L_{v_1}))}.  
\end{align*}

By \cite[Theorem 1]{MR0942522}, the matching condition of $f_G$ and $f_H$ is equivalent to the equality of the numbers of conjugacy orbits in the stable conjugacy orbit of $\gamma$ and $\delta$.  This is immediate from the $H^1$-vanishing theorem.  
\end{prf}

\subsection{Applying a stable trace formula}
Let $\widetilde{\chi}$ be the central character of the automorphic representation $\Pi$.  Thanks to the stabilization of discrete part of the trace formula in \cite{MR1802795}, \cite{MR1954821} and \cite{MR2031854}, we have the equalities
\begin{align*}
	I_{\mathrm{disc}}^H(f)&=S_{\G}(f^G)+\sum_{C\in\mathcal{E}_{\mathrm{ell}}(H),\;C\neq \G}\iota(C,H)S_C(f^C) \\
	I_{\mathrm{disc}}^{\G}(f^G)&=S_{\G}(f^G)+\sum_{C\in\mathcal{E}_{\mathrm{ell}}(\G),\;C\neq \G}\iota(C,\G)S_C((f^G)^C).  
\end{align*}
Here, $C$ runs through the set $\mathcal{E}_{\mathrm{ell}}$ of elliptic endoscopic groups, $\iota$ is some constant, $f^C$ is a transfer to $C$, and $S_C$ is a stable distribution depending only on $C$.  

As we see in the previous subsection, we can take the transfer as $f^G=f$.  In addition, since the orbital integrals of $f_{v_1}^G$ and $f_{v_1}$ are stable from the proof of Lemma~\ref{matching-orbital}, we have $S_C(f^C)=S_C((f^G)^C)=0$ for $C\neq\G$.  Hence, we have $I_{\mathrm{disc}}^H(f)=I_{\mathrm{disc}}^{\G}(f^G)$.  

In our choice of test functions, we take pseudo-coefficients at more than one places ($v_1$ and archimedean ones) and they are supported on elliptic elements, therefore cuspidal in the sense of Arthur's use in \cite[Section 7]{MR0939691}.  The trace formula can be simplified for terms associated to proper Levi subgroups to vanish (see \cite[Theorem 7.1 and Corollary 7.3]{MR0939691}).  For precise description of $I_{\mathrm{disc}}$, see also \cite{MR2192011}.  
By applying this simplified trace formula, varying $f_v$ for all finite $v$ other than $v_1$, and using linear independence of characters, we have the following equation: 
\begin{multline}\label{eq-of-trace}
	\sum_{\substack{\rho\subset L^2_{\mathrm{disc}}(\G(L)\backslash\G(\A_L),\widetilde{\chi}) \\[4pt] \rho^{v_0}\cong \Pi^{v_0}}}
		m(\rho)\Tr \rho_{v_0}(f_{v_0}^G) \\
		=\sum_{\substack{\sigma\subset L^2_{\mathrm{disc}}(H(L)\backslash H(\A_L),\widetilde{\chi}) \\[4pt] \sigma^{v_0,v_1}\cong \Pi^{v_0,v_1},\;\sigma_{v_1}\cong \mu_1\mathrm{St}^H_{v_1}}}
				m(\sigma)\Tr_{\sigma_{v_0}}(f_{v_0}).
\end{multline}

Here, $\widetilde{\chi}$ is the central character of $\Pi$ and $m$ denotes the multiplicity in the discrete spectrum $L^2_{\mathrm{disc}}(-,\widetilde{\chi})$.  

In the following of this paper, we determine the set of the $v_0$ components of irreducible summand of  $L^2_{\mathrm{disc}}(\G(L)\backslash\G(\A_L),\widetilde{\chi})$ appearing on the left hand side of the above equation.  It is denoted by 
\[
	\Sigma=\{\rho_{v_0}\mid \rho\subset L^2_{\mathrm{disc}}(\G(L)\backslash\G(\A_L),\widetilde{\chi}),\; \rho^{v_0}\cong \Pi^{v_0}\}.
\]
In the next section, we show that this set agrees with the local discrete $L$-packet containing $\mu_0\otimes\pi_G$ constructed in \cite{MR3568940}.  

\section{Xu's local and global results for $\mathrm{GSp}_{2n}$}
We first recall Xu's local and global results we need.  We use the same notation as \cite{MR3568940} and \cite{MR4887967}.   In this section, we put $\widetilde{G}=\mathrm{GSp}_{2n}$ and $G=\mathrm{Sp}_{2n}$.  

\subsection{Local discrete $L$-packets}\label{Xus-local}
Let $\phi\in \Phi_{\mathrm{bdd}}(G(F))$ be a bounded local $L$-parameter of $G(F)$ and $\Pi_\phi$ be the Arthur's $L$-packet of $\phi$.  For a central character $\widetilde{\chi}$ of $Z_{\widetilde{G}(F)}$, we write $\widetilde{\Pi}_{\phi,\tilde{\chi}}$ for the set of equivalence classes of irreducible representations $\pi$ of $\widetilde{G}(F)$ such that
	\begin{itemize}
		\item the central character of $\pi$ is $\widetilde{\chi}$, and
		\item the restriction $\pi|_{G(F)}$ is isomorphic to a direct sum of some representations in $\Pi_\phi$ with some multiplicities.  
	\end{itemize}
We note that the second condition is equivalent to that $\pi|_{G(F)}$ contains at least one representation in $\Pi_\phi$.  We put $X=\Hom(\widetilde{G}(F)/Z_{\widetilde{G}}(F)G(F),\C^\times)$.  This group acts on the set $\widetilde{\Pi}_{\phi,\tilde{\chi}}$ and \cite[Proposition 6.28]{MR3568940} says the stabilizers of all representations are the same (in \cite{MR3568940}, it is denoted by $\alpha(\mathcal{S}_{\underline{\phi}})$).  

By \cite[Theorem 6.30]{MR3568940},  there exists a unique subset $\Pi_{\tilde{\phi}}\subset \widetilde{\Pi}_{\phi,\tilde{\chi}}$ up to twisting by $X$ such that
	\begin{enumerate}
		\item $\displaystyle \widetilde{\Pi}_{\phi,\tilde{\chi}}=\bigsqcup_{\omega\in X/\alpha(\mathcal{S}_{\underline{\phi}})}\Pi_{\tilde{\phi}}\otimes \omega$, and 
		\item the sum of trace distributions $\displaystyle \sum_{\pi\in\Pi_{\tilde{\phi}}}\Tr_{\pi}$ is a stable distribution.  
	\end{enumerate}
For each local discrete $L$-parameter $\phi$, we choose and fix an $L$-packet $\Pi_{\tilde{\phi}}$ of $\widetilde{G}(F)$ satisfying the above properties.  

\subsection{Decomposition of the discrete spectrum}
Let $\widetilde{\chi}$ be a character of $Z_{\widetilde{G}}(\A_L)/Z_{\widetilde{G}}(L)$  For a global parameter $\phi=\otimes^\prime_v \phi_v\in\Phi_2(G)$, by \cite[Theorem 1.1]{MR4887967}, there exists a global $L$-packet $\Pi_{\tilde{\phi}}$ of isomorphism classes of irreducible admissible representations of $\widetilde{G}(\A_L)$ such that 
	\begin{enumerate}
		\item for each $\pi=\otimes^\prime_v \pi_v\in\Pi_{\tilde{\phi}}$, $\pi_v\in\Pi_{\tilde{\phi}_v}$ for all places $v$, and 
		\item $\displaystyle L^2_{\mathrm{disc}}(\widetilde{G}(L)\backslash\widetilde{G}(\A_L),\widetilde{\chi})=\bigoplus_{\phi\in\Phi_2(G)}\bigoplus_{\omega\in Y/\alpha(\mathcal{S}_\phi)} \bigoplus_{\substack{\pi\in\Pi_{\tilde{\phi}} \\[4pt] \langle\cdot,\pi\otimes\omega\rangle=1}} \pi\otimes\omega$.  
	\end{enumerate}
Here, $Y=\Hom(\widetilde{G}(\A_L)/\widetilde{G}(L)Z_{\widetilde{G}}(\A_L)G(\A_L),\C^\times)$ whose subgroup $\alpha(\mathcal{S}_\phi)$ acts on $\Pi_{\tilde{\phi}}$ trivially and this packet is unique up to twisting by $Y$.  We now choose and fix a global packet for each global discrete parameter.  

Moreover, for $\phi\in\Phi_2(G)$ and $\omega\in Y$, we put
\[
	R(\phi, \omega)=\bigoplus_{\substack{\pi\in\Pi_{\tilde{\phi}} \\[4pt] \langle\cdot,\pi\otimes\omega\rangle=1}} \pi\otimes\omega.  
\]

We assume that the cuspidal automorphic representation $\Pi$ is a summand of $R(\phi,\omega)$.  We note that in order for the following argument to be valid, it plays a significant role that $\Pi_\infty$ is a discrete series representation and $\Pi_{v_1}$ is a character twist of the Steinberg representation.  

\begin{lem}\label{multiplicity-one-and-kret-shin}
	Suppose an irreducible summand $\rho$ of $R(\psi,\eta)$ satisfies $\rho^{v_0}\cong\Pi^{v_0}$.  Then, we have $\phi=\psi$, $\omega=\eta$ and $m(\rho)=1$.  
\end{lem}

\begin{prf}
	Since $\rho_v\cong \Pi_v$ for all places $v$ except for $v_0$, we have $\phi_v=\psi_v$ for $v\neq v_0$.  By the strong multiplicity one theorem for $\mathrm{GL}_N$, we can see $\phi=\psi$.  Thus, we can write of the form $\Pi=\pi_0\otimes\omega$ and $\rho=\pi\otimes\eta$ for some automorphic $\pi_0,\pi\in\Pi_{\tilde{\phi}}$.  By the same argument in \cite[Theorem 12.1]{MR4556781}, we can see $\omega=\eta$.  Also, \cite[Theorem 12.1]{MR4556781} shows $m(\rho)=1$.  
\end{prf}

Each local $L$-packet of $G(F)$ is invariant under the adjoint action by $\widetilde{G}(F)$.  Let $\pi$ be an irreducible representation of $\widetilde{G}(F)$ whose central character is $\widetilde{\chi}$.  If $\pi_0$ is an irreducible subrepresentation of the restriction $\pi|_{G(F)}$, we can see that $(\widetilde{\chi}\cdot\pi_0)(zg)=\widetilde{\chi}(z)\pi_0(g)\;(z\in Z_{\tilde{G}}(F),\;g\in G(F))$ is an irreducible $Z_{\tilde{G}}(F)G(F)$-subrepresentation of $\pi$.  Frobenius reciprocity shows that there is an injective $Z_{\widetilde{G}}(F)G(F)$-homomorphism $\pi|_{Z_{\widetilde{G}}(F)G(F)}\rightarrow \Res^{\widetilde{G}(F)}_{Z_{\widetilde{G}}(F)G(F)}\Ind^{\widetilde{G}(F)}_{Z_{\tilde{G}}(F)G(F)} (\widetilde{\chi}\cdot\pi_0) $.  On the other hand,  since $Z_{\tilde{G}}(F)G(F)$ is a normal subgroup of $\widetilde{G}(F)$, we have
\[
	 \Res^{\widetilde{G}(F)}_{Z_{\widetilde{G}}(F)G(F)}\Ind^{\widetilde{G}(F)}_{Z_{\tilde{G}}(F)G(F)} (\widetilde{\chi}\cdot\pi_0)\cong\bigoplus_{g\in \widetilde{G}(F)/Z_{\widetilde{G}}(F)G(F)}\widetilde{\chi}\cdot\pi^g_0
\]
by Mackey's formula.  Here, $\pi_0^g(\xi)=\pi_0(g\xi g^{-1})$.  Hence, we can define a map $\mathrm{Res}\colon \widetilde{\Pi}_{\phi_{v_0},\tilde{\chi}}\rightarrow \Pi_{\phi_{v_0}}/\widetilde{G}(F)$ by putting $\mathrm{Res}(\pi)=[\pi_0]$ in the above setting.  

\begin{prop}\label{inj}
	The map $\mathrm{Res}$ induces an injection $\Sigma\rightarrow \Pi_{\phi_{v_0}}/\widetilde{G}(F)$.  In particular, $\Sigma$ is finite.  
\end{prop}

\begin{prf}
	By the strong multiplicity one theorem for Hecke characters, all members of $\Sigma$ has the same central character $\widetilde{\chi}$.  Therefore, Lemma~\ref{multiplicity-one-and-kret-shin} implies the inclusion $\Sigma\subset \widetilde{\Pi}_{\phi_{v_0},\tilde{\chi}}$.  

We will show that $\Res|_\Sigma$ is injective.  Suppose two cuspidal automorphic representations $\rho,\rho^\prime\subset L_\mathrm{disc}^2(\widetilde{G}(L)\backslash\widetilde{G}(\A_L),\widetilde{\chi})$ satisfying $\rho^{v_0}\cong (\rho^\prime)^{v_0}\cong\Pi^{v_0}$ whose $v_0$ components have isomorphic irreducible summands in thier restriction to $\mathrm{Sp}_{2n}(F)$.  That is $\mathrm{Res}(\rho_{v_0})=\mathrm{Res}(\rho^\prime_{v_0})$.  By the proof of Lemma~\ref{multiplicity-one-and-kret-shin}, there exist automorphic $\pi,\pi^\prime\in\Pi_{\tilde{\phi}}$ and $\eta\in Y$ such that $\rho=\pi\otimes\eta$ and $\rho^\prime=\pi^\prime\otimes\eta$.  Because $\pi_{v_0}|_{\mathrm{Sp}_{2n}(F)}$ and $\pi^\prime_{v_0}|_{\mathrm{Sp}_{2n}(F)}$ have isomorphic irreducible summands, \cite[Corollary 6.4]{MR3568940} shows that $\pi_{v_0}$ is a twist of $\pi^\prime_{v_0}$ by a character.  On the other hand, $\pi_{v_0},\pi^\prime_{v_0}\in \Pi_{\tilde{\phi}_{v_0}}$.  By the property (i) of Xu's local $L$-packets in Section~\ref{Xus-local}, we have $\pi_{v_0}\cong\pi^\prime_{v_0}$ hence $\rho_{v_0}\cong \rho^\prime_{v_0}$.  
\end{prf}

\section{Comparison of Xu's local packets and ours}

\subsection{Proof of the main result}

Let $e$ denote the Kottwitz sign of $H(F)$ (cf. \cite{MR0697075}).  We note that $e=(-1)^\frac{n(n+1)}{2}$ in our case.  

\begin{prop}\label{stability-of-our-packet}
	There exists a finite set $\Sigma^H$ of irreducible discrete series representations of $H(F)$ and positive integers $m(\tau)$ for each $\tau\in\Sigma^H$ satisfying
\[
	\sum_{\pi\in\Sigma}\Theta_\pi(g)=e\cdot\sum_{\tau\in\Sigma^H}m(\tau)\Theta_\tau(h)
\]
for all matching strongly regular semisimple elements $g\in \G(F)$ and $h\in H(F)$.  
\end{prop}

\begin{prf}
We choose a specific pair of matching test functions $f_{v_0}^G$ and $f_{v_0}$ as follows.  We note that automorphic representations contributing the both sides of the equation (\ref{eq-of-trace}) are not necessarily finite, but for a fixed pair of matching test functions, they are finite.  For a fixed pair of matching strongly regular semisimple elements $g\in\mathrm{GSp}_{2n}(F)$ and $h\in H(F)$, let $T$ be the maximal torus of $H$ containing $h$ and $T^G$ be the maximal torus of $\G$ containing $g$.  We define two maps, which are finite-to-one, $\phi^G$ and $\phi$ by
\[
	\phi^G\colon \G(F)/T^G(F)\times {T^G(F)}^\mathrm{srs}\rightarrow \G(F)^{\mathrm{srs}};\;(\xi,t)\mapsto \xi t\xi^{-1},
\]
and 
\[
	\phi\colon H(F)/T(F)\times {T(F)}^\mathrm{srs}\rightarrow H(F)^{\mathrm{srs}};\;(\eta,s)\mapsto \eta s\eta^{-1}.
\]

Here we write $(\cdot)^{\mathrm{srs}}$ for the subset of the strongly regular semisimple elements.  Now we take compact open subsets $C^G\subset \G(F)/T^G(F)$ and $C\subset H(F)/T(F)$ such that  intersections with thier rational Weyl groups $C^G\cap W_F(\G,T^G)$ and $C\cap W_F(H,T)$ are trivial.  Here, the rational Weyl groups are defined as $W_F(\G,T^G)=N_{\G(F)}(T^G(F))/T^G(F)$.  We take Haar measures on $G(F)$ and $H(F)$ such that $\mathrm{vol}(C^G)=\mathrm{vol}(C)=1$.  Since $g$ and $h$ are matching, there exists an isomorphism $\psi\colon H(\overline{F})\rightarrow \mathrm{GSp}_{2n}(\overline{F})$ which maps $h$ to $g$ and of course $\psi(T)=T^G$.  Via this isomorphism, we set Haar measures on $T^G(F)$ and $T(F)$ which are compatible.  We also take compact open neighborhoods $V^G$ of $g$ in $T^G(F)$ and $V$ of $h$ in $T(F)$ which satisfy
\begin{itemize}
	\item for any Weyl elemement $w\in W_F(\G,T^G)$, $wV^G\cap V^G=\emptyset$, 
	\item for any Weyl elemement $w\in W_F(H,T)$, $wV\cap V=\emptyset$, 
	\item the character functions $\Theta_{\pi}$ are constant on $V^G$ for each $\pi\in \Sigma$, 
	\item the norm of Weyl discriminant of $\G$ denoted by $D_{\G}^2$ is constant on $V^G$, and
	\item the norm of Weyl discriminant of $H$ denoted by $D_H^2$ is constant on $V$.  
\end{itemize}
Moreover, we can take them as $\psi(V)=V^G$ and in this situation, the volumes of them are equal by our setting of Haar measures.  

Now we put $f^G_{v_0}$ to be the characteristic function of $\phi^G(C^G\times V^G)$ multiplied by $|W_F(\G,T^G)|$ and $f_{v_0}$ to be that of $\phi(C\times V)$ multiplied by $e|W_F(H,T)|$.  Then, it is easy to see that $f^G_{v_0}$ and $f_{v_0}$ have matching orbital integrals.  This is because $\delta\in\phi(C\times V)$ is matching to $\gamma\in \mathrm{GSp}_{2n}(F)^{\mathrm{srs}}$ if and only if the stable conjugacy orbit of $\gamma$ intersects $\phi^G(C^G\times V^G)$ and because in this case, the intersection of the stable conjugacy orbit of $\gamma$ (resp. $h$) and the support of $f^G_{v_0}$ (resp. $f_{v_0}$) has just one element.  

For this pair of matching test function, we write $\Sigma^H$ for the set of $v_0$-component of the automorphic representations contributing the right hand sides of the equation (\ref{eq-of-trace}).  From the Weyl integration formula, thr left hand side of the equation (\ref{eq-of-trace}) for the above test function $f_{v_0}^G$ is 
\[
	\sum_{\pi\in\Sigma}\frac{1}{|W_F(\G,T^G)|}\int_{T^G(F)}D_{\G}(t)^2\biggl(\int_{\G(F)/T^G(F)}f_{v_0}^G(\xi t\xi^{-1})\Theta_\pi(t)d\xi\biggr)dt.  
\]
In our setting, $\displaystyle \frac{1}{|W_F(\G,T^G)|}\int_{\G(F)/T^G(F)}f_{v_0}^G(\xi t\xi^{-1})\Theta_\pi(t)d\xi=\Theta_\pi(t)$.  
From a similar computation for the right hand side of the equation (\ref{eq-of-trace}), we have 
\[
	\sum_{\pi\in\Sigma}\mathrm{vol}(V^G)\Theta_{\pi}(g)=e\sum_{\tau\in\Sigma^H}\int_{T(F)}\chi_V(s)m(\tau)\Theta_{\tau}(s)ds
\]
Here, $\chi_V$ denotes the characteristic function on $V$ and $m(\tau)$ are some positive integers.  We note that the Weyl discriminants of $\G$ and $H$ are the same since $H$ is an inner form of $\G$.  We shrink $V$ and $V^G$ as for all $\pi^H\in\Sigma^H$, $\Theta_{\pi^H}$ are constant on $V$.  Then, we have the equation as stated in the proposition.  

By the same argument in \cite[Section 21]{MR0876160}, that is, considering the central exponents and applying Casselman's criterion for square integrabilitiy, we can see that all members of $\Sigma^H$ are discrete series representations since those of $\Sigma$ are so.  
\end{prf}

\begin{prop}\label{surj}
	For a pair of a parameter and a character $(\phi,\omega)$ fixed before Lemma~\ref{multiplicity-one-and-kret-shin}, we have $\Sigma=\Pi_{\tilde{\phi}_{v_0}}\otimes\omega_{v_0}$.  
\end{prop}

\begin{prf}
	By Proposition~\ref{stability-of-our-packet}, the sum of characters $\displaystyle \Theta_\Sigma=\sum_{\pi\in\Sigma}\Theta_\pi$ is stably invariant.  Hence, one can see that the sum of trace distributions $\displaystyle \mathrm{Tr}_\Sigma=\sum_{\pi\in\Sigma} \mathrm{Tr}_\pi$ is a stable distribution.  This is due to the stable version of the Weyl integration formula.  For a local test function $f_{v_0}^G$ such that $SO(f_{v_0}^G)$ vanishes at every strongly regular semisimple element, the stable version of the formula induces the following equation: 
\[
	\mathrm{Tr}_\Sigma(f_{v_0}^G)=\sum_{T^G}\frac{1}{|W(T^G,\G)(F)|}\int_{{T^G(F)}^\mathrm{srs}}D_{\G}(t)^2 SO_t(f_{v_0}^G) \Theta_\Sigma(t) dt.  
\]
Here, $SO$ denotes the stable orbital integral (which is not normalized by the quare root of the norm of Weyl discriminant in this paper), i.e. 
\[
	SO_t(f_{v_0}^G)=\int_{{\G}(F)/{\G}(F)_t} f_{v_0}^G(\xi t\xi^{-1})d\xi.
\]
Moreover, $T^G$ runs through the stable conjugacy classes of maximal tori of $\G$ and $W(T^G,\G)=N_{\G}(T^G)/T^G$.  Since the stable orbital integral vanishes, $\mathrm{Tr}_\Sigma(f_{v_0}^G)=0$, hence $\mathrm{Tr}_\Sigma$ is stable.  

By restricting the function $\Theta_\Sigma$ to $\mathrm{Sp}_{2n}(F)$, we can have a stably invarinat function which is a linear combination of characters of some members in $\Pi_{\phi_{v_0}}$.  On the other hand, each discrete $L$-packets of $\mathrm{Sp}_{2n}(F)$ is atomically stable, that is a some linear combination of trace characters of the all members in each $L$-packet and no proper subset of the packet does not have this property (see \cite[Proposition 7.2.2]{some-comment}).  Therefore, the map $\mathrm{Res}$ in Proposition~\ref{inj} is also surjective.  By \cite[Corollary 6.4]{MR3568940} and Proposition~\ref{inj}, we can see 
\[
	 \widetilde{\Pi}_{\phi,\tilde{\chi}}=\bigsqcup_{\omega\in X/\alpha(\mathcal{S}_{\underline{\phi}})}\Sigma\otimes \omega.
\]
From the uniqueness up to character twists of the Xu's local packets satisfying two properties (i) and (ii) in Section 3.1, the proposition follows.  
\end{prf}

\begin{thm}\label{main}
	Let $\pi_G$ be an irreducible discrete series representation of $\G(F)$ and $\Sigma_0$ be the discrete $L$-packet containing $\pi_G$ defined by Xu.  Then, there exist a finite set $\Sigma_0^H$ of isomorphism classes of irreducible discrete series representations of $\mathrm{GU}_n(D)$ and positive integers $m(\tau)$ for each $\tau\in\Sigma_0^H$ such that 
\[
	\sum_{\pi\in\Sigma_0} \Theta_\pi(g)=e\cdot\sum_{\tau\in\Sigma_0^H}m(\tau)\Theta_\tau(h)
\]
for any pair of matching strongly regular semisimple elements $g\in\G(F)$ and $h\in\mathrm{GU}_n(D)$.
\end{thm}

\begin{prf}
	By Proposition~\ref{surj}, $\Sigma_0=\Pi_{\tilde{\phi}_{v_0}}\otimes \omega_{v_0}\mu_0^{-1}=\Sigma\otimes\mu_0^{-1}$.  Hence, if we take a discrete packet $\Sigma^H$ of $H(F)$ satisfying the property of Proposition~\ref{stability-of-our-packet} and if we put $\Sigma_0^H=\Sigma^H\otimes\mu_0^{-1}$, the character relation holds.  
\end{prf}

By restricting the above character relation to $\mathrm{Sp}_{2n}(F)$ and its inner form $\mathrm{U}_n(D)$, we have the following analogous result for $\mathrm{Sp}_{2n}$.  

\begin{cor}
	For every discrete $L$-packet $\Pi_\psi$ of $\mathrm{Sp}_{2n}(F)$, there exist a finite set $\Pi_\psi^H$ of isomorphism classes of irreducible discrete series representations of $\mathrm{U}_n(D)$ and positive integers $m^\prime(\pi)$ and $m^\prime(\tau)$ such that 
\[
	\sum_{\pi\in\Pi_\psi} m^\prime(\pi)\Theta_\pi(g)=e\cdot\sum_{\tau\in\Pi_\psi^H}m^\prime(\tau)\Theta_\tau(h)
\]
for any pair of matching strongly regular semisimple elements $g\in\mathrm{Sp}_{2n}(F)$ and $h\in\mathrm{U}_n(D)$.
\end{cor}

\subsection{Generic representations in a discrete $L$-packet of $\G(F)$}
We now remark on generic representations in an $L$-packet of $\G(F)$.  More results for generic representations of $\G(F)$ from the view of restriction to $\mathrm{Sp}_{2n}$, see \cite[Section 6]{MR3568940}.  

Let $\widetilde{B}=\widetilde{T}N$ be the group of $F$-valued points of a Borel subgroup of $\G$ with its Levi decomposition.  Note that $N$ is a maximal unipotent subgroup of both $\G(F)$ and $\mathrm{Sp}_{2n}(F)$

\begin{lem}\label{existence-of-generic-reps}
	Let $\psi$ be a genric  character of $N$.  Then, $\pi$ is an irreducible generic representation of $\G(F)$ with respect to $(N,\psi)$ if and only if its restriction $\pi|_{\mathrm{Sp}_{2n}(F)}$ contains an irreducible generic subrepresentation of $\mathrm{Sp}_{2n}(F)$ with respect to $(N,\psi)$ with multiplicity one.  
\end{lem}

\begin{prf}
	Let $\pi|_{\mathrm{Sp}_{2n}(F)}=\bigoplus_{i=1}^r \tau_i$ be a decomposition into a direct sum of irreducible subrepresentations.  If $\pi$ is generic, we have 
\[
	\C\cong \Hom_N(\pi,\psi)\cong\bigoplus_{i=1}^r \Hom_N(\tau_i,\psi).  
\]
By the uniqueness of Whittaker model, there exists a unique $i$ such that $\tau_i$ is generic with respect to $(N,\psi)$.  

Conversely, we suppose that $\tau_1$ is generic.  Then, the composite of a non-zero $N$-homomorphism $\tau_1\rightarrow \psi$ and the projection $\pi\rightarrow \tau_1$ is a non-zero $N$-homomorphism $\pi\rightarrow \psi$.  Hence, $\pi$ is generic with respect to $(N,\psi)$.  
\end{prf}

We note that $N$ is invariant under conjugation by $\widetilde{T}$.  The same argument as the following lemma can be seen in \cite{MR3681684}.  

\begin{lem}\label{twist-generic}
	Let $\kappa\in\widetilde{T}$, $\psi$ be a generic character of $N$, and $\tau$ be an irreducible representation of $\mathrm{Sp}_{2n}(F)$.  Then, in order for $\tau$ to be generic with respect to $(N,\psi)$, it is necessary and sufficient that $\tau^\kappa$ is generic with respect to $(N,\psi^\kappa)$.  
\end{lem}

\begin{prf}
	We write $V$ (resp. $V^\tau$) for the representation spaces of $\tau$ (resp. $\tau^\kappa$).  Since $N\rightarrow N;\:n\mapsto \kappa n\kappa^{-1}$ is bijective, the following equation of coinvariant spaces holds: 
\[
	V_\psi=(V^\kappa)_{\psi^\kappa}.  
\]
Here, for example, $V_\psi=V/\langle \pi(u)v-\psi(u)v\mid v\in V,\;u\in N\rangle$.  The lemma follows from the isomorphism $\Hom_\C(V_\psi,\C)\cong \Hom_N(V,\psi)$ and the same isomrphism for $V^\kappa$ and $\psi^\kappa$.  
\end{prf}

We put $T=\widetilde{T}\cap \mathrm{Sp}_{2n}(F)$.  Since $N$ is normal in $\widetilde{B}$ 
, we need to consider only $T$-adjoint (resp. $\widetilde{T}$-adjoint) orbits of a fixed generic character of $N$ for generic representations of $\mathrm{Sp}_{2n}(F)$ (resp. $\mathrm{GSp}_{2n}(F)$).  

\begin{lem}\label{num-of-orbits}
	For any generic character $\psi$ of $N$, we can take a complete system $\{\kappa_i\}$ of representatives of $\G(F)/Z_{\G}(F)\mathrm{Sp}_{2n}(F)$ such that $\{\psi^{\kappa_i}\}$ is a complete system of the $T$-adjoint orbits of generic characters on $N$.  
\end{lem}

\begin{prf}
	We take $\widetilde{B}$ as the subgroup of the upper triangular matrices.  Then, a generic character of $N$ is of the form
\[
	\psi_{F;\;a_1,\ldots,a_n}\colon N\rightarrow \C;\;(u_{ij})\mapsto \psi_F(a_1u_{1\,2}+\cdots+a_{n}u_{n\,n+1})
\]
where $\psi_F\colon F\rightarrow \C$ is a nontrivial additive character and $a_i\in F^\times$.  For $t=\mathrm{diag}(t_1,\ldots,t_n,t_n^{-1},\ldots,t_1^{-1})\in T$, we have 
\[
	\psi_{F;\;a_1,\ldots,a_n}^t=\psi_{F;\;t_1^{-1}t_2a_1,\ldots,t_{n-1}^{-1}t_na_{n-1},t_n^{-2}a_n}.  
\]
Hence, the number of orbits is equal to $\#(F^\times/(F^\times)^2)$.  Since the determinat map induces an isomorphism $\mathrm{GSp}_{2n}(F)/Z_{\G}(F)\mathrm{Sp}_{2n}(F)\cong F^\times/(F^\times)^2$, the number is also equal to $\#(\mathrm{GSp}_{2n}(F)/Z_{\G}(F)\mathrm{Sp}_{2n}(F))$.  This implies the lemma.  
\end{prf}

\begin{prop}
	For each discrete $L$-packet $\widetilde{\Pi}$ of $\G(F)$, there exists a unique generic representation $\pi_0$ with respect to $(N.\psi)$.  Moreover, the restriction $\pi_0|_{\mathrm{Sp}_{2n}(F)}$ is a direct sum of generic representations of $\mathrm{Sp}_{2n}(F)$ for some generic character of $N$: 
\[
	\pi_0|_{\mathrm{Sp}_{2n}(F)}=\tau_1\oplus\cdots\oplus\tau_r
\]
and for each generic character $\psi^\prime$ of $N$, there uniquely exists $i$ such that $\tau_i$ is $\psi^\prime$-generic.  
\end{prop}

\begin{prf}
	We write $\widetilde{\chi}$ for the central character of $\pi$.  Suppose $\pi$ is generic with respect to a generic character $\psi$.  For an irreducible $\psi$-generic subrepresentation $\tau\subset\pi|_{\mathrm{Sp}_{2n}(F)}$ (cf. Lemma~\ref{existence-of-generic-reps}), as the argument in defining our map $\mathrm{Res}$ before Proposition~\ref{inj}, we can see that $\pi|_{\mathrm{Sp}_{2n}(F)}$ is a subrepresentation of $\bigoplus_{g\in \G(F)/Z_{\G}(F)G(F)}\widetilde{\chi}\cdot\tau^g$.  Hence, the proposition follows from Lemma~\ref{twist-generic} and Lemma~\ref{num-of-orbits}.  Uniqueness is due to that of generic representations in a discrete $L$-packet of $\mathrm{Sp}_{2n}(F)$.  
\end{prf}

\subsection{Atomic stability of packets on the inner form side}

We first remind the definition of an elliptic inner product on the space of elliptic (stable) orbital integrals.  Let $\mathbb{M}$ be a reductive group over $F$ and $\mathcal{I}_e(M)$ be the space of orbital integrals on elliptic regular elements of $M=\mathbb{M}(F)$.  For $\Phi, \Psi\in \mathcal{I}_e(M)$, we put
\[
	\langle \Phi,\Psi \rangle_{\mathbb{M},e}=\sum_{T}\frac{1}{|W_F(M,T)|}\int_{{Z_M}\backslash T}D_M(t)^2\Phi(t)\overline{\Psi(t)}dt, 
\]
where $T$ runs over the conjugacy classes of elliptic maximal tori of $M$, $W$ denotes the Weyl groups, and $Z$ denotes the center.  

We also need another version of a pairing for stable ones.  Let $\mathcal{SI}_e(M)$ be the space of stable orbital integrals on elliptic regular elements of $M=\mathbb{M}(F)$.  For $\Phi, \Psi\in \mathcal{SI}_e(M)$, we put
\[
	\langle \Phi,\Psi \rangle^\mathrm{st}_{\mathbb{M},e}=\sum_{\mathbb{T}}\frac{1}{|W(\mathbb{M},\mathbb{T})(F)|}\int_{Z_\mathbb{M}(F)\backslash\mathbb{T}(F)}\frac{1}{n(t)}D_M(t)^2\Phi(t)\overline{\Psi(t)}dt, 
\]
where $\mathbb{T}$ runs over the stable conjugacy classes of elliptic maximal tori of $\mathbb{M}$ and $n(t)$ is the number of conjugacy classes in the stable conjugacy class of $t$.  

For $\Phi\in\mathcal{I}_e(M)$, we write $\displaystyle S\Phi(t)=\sum_{t^\prime}\Phi(t^\prime)$ where $t^\prime$ runs through the conjugacy classes in the stable conjugacy class of $t$.  By the definition, we can see $\langle S\Phi,S\Psi \rangle^\mathrm{st}_{\mathbb{M},e}=\langle \Phi,\Psi \rangle_{\mathbb{M},e}$ for stable functions $\Phi, \Psi\in \mathcal{I}_e(M)$.  

\begin{lem}\label{no-crossing}
	For any distinct discrete $L$-packets $\Pi_1$ and $\Pi_2$ of $\G(F)$, the corresponding packets of $H(F)$ by Theorem~\ref{main} have no isomorphic representations: 
	\[
		\Pi_1^H\cap \Pi_2^H=\emptyset.  
	\]
\end{lem}

\begin{prf}
	We may assume that all members in $\Pi_1$ and $\Pi_2$ (hence also in $\Pi_1^H$ and $\Pi_2^H$) are unitary.  We write the stable sum of characters in a packets as
	\[
		\Theta_{\Pi_i}=\sum_{\pi\in\Pi_i}\Theta_\pi,\quad \Theta_{\Pi_i^H}=\sum_{\pi^H\in\Pi_i^H}m(\pi^H)\Theta_{\pi^H}.
	\]
	Thanks to the canonical identification of stable conjugacy classes of elliptic maximal tori of $\G$ and $H$, we have the following equation from the character relation: 
	\[
		\langle \Theta_{\Pi_1}, \Theta_{\Pi_2}\rangle^{\mathrm{st}}_{\G,e}=\langle \Theta_{\Pi_1^H}, \Theta_{\Pi_2^H}\rangle^{\mathrm{st}}_{H,e}.
	\]
	Since each sum of characters is stable and the number of conjugacy classes in a stable class is invariant under inner twists, we have 
	\[
		\langle \Theta_{\Pi_1}, \Theta_{\Pi_2}\rangle_{\G,e}=\langle \Theta_{\Pi_1^H}, \Theta_{\Pi_2^H}\rangle_{H,e}.
	\]
	Now $\Pi_1$ and $\Pi_2$ have no common representations.  Then, by the orthogonality of characters, the LHS is zero.    If there was a common representation $\pi^H\in\Pi_1^H\cap\Pi_2^H$, then the RHS would be equal to or greater than $m(\pi^H)^2>0$.  This is a contradiction.  
\end{prf}

\begin{prop}
	Let $\Pi$ be a discrete $L$-packet of $\G(F)$ and $\Pi^H$ be the corresponding packet of $H(F)$ by Theorem~\ref{main}.  Then, for each proper subset $\Pi_0^H\subsetneq \Pi^H$, any nonzero linear combination of characters of members in $\Pi_0^H$ is not stable.  
\end{prop}

\begin{prf}
	Suppose $\Pi_0^H\subsetneq \Pi^H$ and $\displaystyle \sum_{\tau\in\Pi_0^H}c_\tau\Theta_\tau$ is stable for some $c_\tau\in\C$.  Then, by a similar argument th the proof of Proposition~\ref{surj}, 
\[
	\mathrm{Tr}_{\Pi_0^H}:=\sum_{\tau\in\Pi_0^H}c_\tau\mathrm{Tr}_\tau
\]
is a stable distribution.  On the other hand, the space of elliptic stable distributions $SD_{\mathrm{ell}}(H(F))$ (which can be regarded as a dual space of $\mathcal{SI}_e(H(F))$) has a basis 
\[
	\biggl\{\mathrm{Tr}_{\Sigma_0^H}:=\sum_{\tau\in\Sigma_0^H}m(\tau)\mathrm{Tr}_\tau\mid \Sigma_0:\;\text{discrete $L$-packets of $\G(F)$}\biggr\}
\]
from \cite[Corollary 3.2.7]{some-comment} and \cite[Proposition 7.2.2]{some-comment}.  Therefore, as an element of $SD_{\mathrm{ell}}(H(F))$, we can write uniquely in the form
\[
	\mathrm{Tr}_{\Pi_0^H}=\sum_{i=1}^rd_i\mathrm{Tr}_{\Sigma_i^H}
\]
where $\Sigma_i$ are distinct discrete $L$-packets of $\G(F)$ and $d_i\in\C^\times$.  Since the trace characters of irreducible discrete representations of $H(F)$ are linearly independent, we have 
\[
	\Pi_0^H=\bigsqcup_{i=1}^r \Sigma_i^H
\]
by Lemma~\ref{no-crossing}.  Hence, Lemma~\ref{no-crossing} shows $r=1$ and $\Pi_0^H=\Sigma_1^H=\Pi^H$.  This contradicts to $\Pi_0^H\neq \Pi^H$.  
\end{prf}

\section{Simple supercuspidal $L$-packets}

\begin{lem}\label{single}
	If the discrete packet $\Sigma_0$ of $\G(F)$ in Theorem~\ref{main} is singleton, the corresponding discrete packet $\Sigma_0^H$ of $H(F)$ is also singleton.  
\end{lem}

\begin{prf}
	Since the central characters of all members in $\Sigma_0$ and $\Sigma_0^H$ are the same, we can assume that all members in $\Sigma_0$ and $\Sigma_0^H$ are unitary.  Assume that $\Sigma_0^H$ contains distinct two discrete series representations $\sigma_1$ and $\sigma_2$.  If $f_0$ is the sum of pseudo-coefficients of $\sigma_1$ and $\sigma_2$, by the equation (\ref{eq-of-trace}), we have the inequality 
\[
	\bigl(\Tr\pi_G(f_0^G)\bigr)^2\geq \bigl(m(\sigma_1)+m(\sigma_2)\bigr)^2\geq 4
\]
for a transfer $f_0^G$ of $f_0$ to $\G(F)$.  

On the other hand, we have the following inequalities for the elliptic pairing: 
\begin{align*}
	\bigl(\Tr\pi_G(f_0^G)\bigr)^2&=\bigl\langle\Theta_{\pi_G},O(f_0^G)\bigr\rangle_{\G,e}^2  \\
		&\leq \bigl\langle\Theta_{\pi_G},\Theta_{\pi_G}\bigr\rangle_{\G,e}\bigl\langle O(f_0^G),O(f_0^G)\bigr\rangle_{\G,e} \\
		&=\bigl\langle O(f_0^G),O(f_0^G)\bigr\rangle_{\G,e} \\
		&=\bigl\langle SO(f_0^G), SO(f_0^G)\bigr\rangle_{\G,e}^{\mathrm{st}} \\
		&\leq\bigl\langle O(f_0), O(f_0)\bigr\rangle_{H,e}=2.
\end{align*}
Here we use the Weyl integration formula and the fact that pseudo-coefficients of discrete series representations are supported on elliptic elements in the first equality, and the second inequality is due to the Caucy-Schwartz inequality.  The third equality follows from the fact that the characterso of unitary representations form a normal orthogonal basis with respect to the elliptic pairing.  Moreover, our remark mentioned just before Lemma~\ref{no-crossing} shows the fourth equality, and we use (iv) of \cite[Theorem 4.3]{MR3267112} to get the fifth inequality.  

Hence, we have $4\leq \bigl(\Tr\pi_G(f_0)\bigr)^2\leq2$ and this is a contradiction.  
\end{prf}

\begin{prop}
	Suppose $\pi_G$ in Theorem~\ref{main} is a simple supercuspidal representation of $\G(F)$.  Then, both of the $L$-packet $\Sigma_0$ containing $\pi_G$ and the corresponding packet $\Sigma_0^H$ of $\mathrm{GU}_n(D)$ are singleton.  
\end{prop}

\begin{prf}
	By \cite[Theorem 5.1]{MR4756397} and \cite[Proposition 5.2]{MR4756397}, the simple supercuspidal $L$-packets of $\mathrm{Sp}_{2n}(F)$ consists of a unique adjoint orbit.  Hence, by Proposition~\ref{inj}, $\Sigma$ is singleton.  Since $\Sigma_0$ has the same cardinality as $\Sigma$, Lemma~\ref{single} shows the proposition.  
\end{prf}

\def\cftil#1{\ifmmode\setbox7\hbox{$\accent"5E#1$}\else
  \setbox7\hbox{\accent"5E#1}\penalty 10000\relax\fi\raise 1\ht7
  \hbox{\lower1.15ex\hbox to 1\wd7{\hss\accent"7E\hss}}\penalty 10000
  \hskip-1\wd7\penalty 10000\box7}
  \def\cftil#1{\ifmmode\setbox7\hbox{$\accent"5E#1$}\else
  \setbox7\hbox{\accent"5E#1}\penalty 10000\relax\fi\raise 1\ht7
  \hbox{\lower1.15ex\hbox to 1\wd7{\hss\accent"7E\hss}}\penalty 10000
  \hskip-1\wd7\penalty 10000\box7}
  \def\cftil#1{\ifmmode\setbox7\hbox{$\accent"5E#1$}\else
  \setbox7\hbox{\accent"5E#1}\penalty 10000\relax\fi\raise 1\ht7
  \hbox{\lower1.15ex\hbox to 1\wd7{\hss\accent"7E\hss}}\penalty 10000
  \hskip-1\wd7\penalty 10000\box7}
  \def\cftil#1{\ifmmode\setbox7\hbox{$\accent"5E#1$}\else
  \setbox7\hbox{\accent"5E#1}\penalty 10000\relax\fi\raise 1\ht7
  \hbox{\lower1.15ex\hbox to 1\wd7{\hss\accent"7E\hss}}\penalty 10000
  \hskip-1\wd7\penalty 10000\box7} \def\cprime{$'$} \def\cprime{$'$}
  \newcommand{\dummy}[1]{}
\providecommand{\bysame}{\leavevmode\hbox to3em{\hrulefill}\thinspace}
\providecommand{\MR}{\relax\ifhmode\unskip\space\fi MR }
\providecommand{\MRhref}[2]{%
  \href{http://www.ams.org/mathscinet-getitem?mr=#1}{#2}
}
\providecommand{\href}[2]{#2}


\begin{thebibliography}{SGA4$\frac12$}

\bibitem[Art88]{MR0939691}
J.~G. Arthur, The invariant trace formula. II. Global theory, J. Amer. Math. Soc. {\bf 1} (1988), no.~3, 501--554.

\bibitem[Art01]{MR1802795}
J.~G. Arthur, A stable trace formula. II. Global descent, Invent. Math. {\bf 143} (2001), no.~1, 157--220.

\bibitem[Art02]{MR1954821}
J.~G. Arthur, A stable trace formula. I. General expansions, J. Inst. Math. Jussieu {\bf 1} (2002), no.~2, 175--277.

\bibitem[Art03]{MR2031854}
J.~G. Arthur, A stable trace formula. III. Proof of the main theorems, Ann. of Math. (2) {\bf 158} (2003), no.~3, 769--873.

\bibitem[Art05]{MR2192011}
J.~G. Arthur, An introduction to the trace formula, in {\it Harmonic analysis, the trace formula, and Shimura varieties}, 1--263, Clay Math. Proc., 4, Amer. Math. Soc., Providence, RI, 2005.

\bibitem[Art13]{MR3135650}
J.~Arthur, \emph{{\dummy{A}}{T}he endoscopic classification of representations:
  Orthogonal and symplectic groups}, American Mathematical Society Colloquium
  Publications, vol.~61, American Mathematical Society, Providence, RI, 2013.

\bibitem[ABPS16]{MR3618046}
A.-M.~Aubert et al., Depth and the local Langlands correspondence, in {\it Arbeitstagung Bonn 2013}, 17--41, Progr. Math., 319, Birkh\"auser/Springer, Cham, 2016.

\bibitem[Bad02]{MR1951441}
A.~I.~Badulescu, Correspondance de Jacquet-Langlands pour les corps locaux de caract\'eristique non nulle, Ann. Sci. \'Ecole Norm. Sup. (4) {\bf 35} (2002), no.~5, 695--747.

\bibitem[CG15]{MR3267112}
P.-S.~Chan and W.~T.~Gan, \emph{{The local Langlands conjecture for ${\rm GSp}(4)$ III: Stability and twisted endoscopy}},  J. Number Theory {\bf 146} (2015), 69--133.

\bibitem[DKV84]{MR0771672}
P. Deligne, D.~A. Kazhdan and M.-F. Vign\'eras, Repr\'esentations des alg\`ebres centrales simples $p$-adiques, in {\it Representations of reductive groups over a local field}, 33--117, Travaux en Cours, Hermann, Paris, 1984.

\bibitem[FK86]{MR0876160}
Y.~Z. Flicker and D.~A. Kazhdan, Metaplectic correspondence, Inst. Hautes \'Etudes Sci. Publ. Math. No. 64 (1986), 53--110.

\bibitem[GT11]{MR2800725}
W.~T. Gan and S. Takeda, The local Langlands conjecture for ${\rm GSp}(4)$, Ann. of Math. (2) {\bf 173} (2011), no.~3, 1841--1882.

\bibitem[GT14]{MR3214276}
W.~T. Gan and W. Tantono, The local Langlands conjecture for $\rm GSp(4)$, II: The case of inner forms, Amer. J. Math. {\bf 136} (2014), no.~3, 761--805.

\bibitem[JL70]{MR0401654}
H.~Jacquet and R.~P.~Langlands, {\it Automorphic forms on ${\rm GL}(2)$}, Lecture Notes in Mathematics, Vol. 114, Springer, Berlin-New York, 1970.

\bibitem[Kal]{MR4680348}
T.~Kaletha, \emph{{Representation of reductive groups over local fields}}, in {\it ICM---International Congress of Mathematicians. Vol. 4. Sections 5--8}, 2948--2975, EMS Press, Berlin.

\bibitem[Kot83]{MR0697075}
R.~E. Kottwitz, Sign changes in harmonic analysis on reductive groups, Trans. Amer. Math. Soc. {\bf 278} (1983), no.~1, 289--297.

\bibitem[Kot88]{MR0942522}
R.~E.~Kottwitz, \emph{{Tamagawa numbers}}, Ann. of Math. (2) {\bf 127} (1988), no.~3, 629--646.

\bibitem[KS23]{MR4556781}
A.~Kret and S.~W.~Shin, \emph{{Galois representations for general symplectic groups}}, J. Eur. Math. Soc. (JEMS) {\bf 25} (2023), no.~1, 75--152.

\bibitem[Oi24]{MR4756397}
M.~Oi, \emph{{Simple supercuspidal L-packets of quasi-split classical groups}}, Mem. Amer. Math. Soc. {\bf 297} (2024), no.~1483, v+161 pp.

\bibitem[Shi12]{MR3004076}
S.~W.~Shin, \emph{{Automorphic Plancherel density theorem}}, Israel J. Math. {\bf 192} (2012), no.~1, 83--120.

\bibitem[Var]{some-comment}
S.~Varma,
\emph{{Some comments on the stable Bernstein center}},
preprint, available at
\url{https://mathweb.tifr.res.in/~sandeepv/stable_center_classical_groups.pdf}.

\bibitem[Xu16]{MR3568940}
B.~Xu, \emph{{On a lifting problem of L-packets}}, Compos. Math. {\bf 152} (2016), no.~9, 1800--1850.

\bibitem[Xu18]{MR3747484}
B.~Xu, \emph{{L-packets of quasisplit $GSp(2n)$ and $GO(2n)$}}, Math. Ann. {\bf 370} (2018), no.~1-2, 71--189.

\bibitem[Xu25]{MR4887967}
B.~Xu, \emph{{Global $L$-packets of quasisplit ${\rm GSp}(2n)$ and ${\rm GO}(2n)$}}, Amer. J. Math. {\bf 147} (2025), no.~2, 401--464.

\bibitem[Zha17]{MR3681684}
Q. Zhang, A local converse theorem for $\rm U(1,1)$, Int. J. Number Theory {\bf 13} (2017), no.~8, 1931--1981.

\end{thebibliography}
\end{document}